\documentclass{amsart}

\usepackage{amssymb, amsmath, latexsym, a4}
\usepackage[latin1]{inputenc}
\usepackage[all]{xy}

\def\xto#1{\xrightarrow[]{#1}}

\def \G{\mathop{\Gamma }\nolimits}

\def \Y{\mathcal  Y}

\newcommand{\Ga}{\raisebox{0.1mm}{$\Gamma$}}

\def\ll{ \lambda}
\def\H {\mathcal H}
\def\Ab{\sf Ab}
\def\der{{\sf Der}}
\def\id{\sf Id}

\def\c {{\bf C}}
\def\d {{\bf D}}

\newcommand{\Z}{\mathbb{Z}}
\def\t{\otimes }
\def\ee{\epsilon }

\def \Coker{\mathop{\sf Coker}\nolimits}
\def \Hom{\mathop{\sf Hom}\nolimits}

\newtheorem{De}{Definition}[section]
\newtheorem{Th}[De]{Theorem}
\newtheorem{Pro}[De]{Proposition}
\newtheorem{Le}[De]{Lemma}
\newtheorem{Co}[De]{Corollary}
\newtheorem{Rem}[De]{Remark}

\begin{document}
\title{Functor homology and homology of commutative monoids}
\bigskip

\author[R.  Kurdiani]{Revaz  Kurdiani}

\address{Ivane Javakhishvili Tbilisi State University,  Georgia} 
\thanks{Research was partially supported by the GNSF Grant ST08/3-387. The second author 
acknowledgement  the support given by the University of Leicester in granting academic study leave}

\author[T. Pirashvili]{Teimuraz  Pirashvili}

\address{Department of Mathematics\\
University of Leicester\\
University Road\\
Leicester\\
LE1 7RH, UK} \email{tp59-at-le.ac.uk}

\maketitle

The aim of this work is to clarify the realtionship between homology theory of commutative monoids constructed \'a la Quillen \cite{ha}, \cite{Q} 
and technology of $\Gamma$-modules as it was developped in \cite{hodge},\cite{aq}, \cite{dold}, \cite{robi}.

In Section \ref{gc} we recall the basics of $\Gamma$-modules and relation with commuttaive algebra (co)homology.
In Section \ref{DDC} we construct an analogue of K\"ahler differentials for commutative monoids. In Section \ref{cmc}  we construct homology theory  for commutative monoids which and then  we prove that commutative monoid homology are  particular case of functor homology develiped in \cite{aq}.

It should be pointout that the cohomology theory of commutative monoids first was constructed by P.-A. Grillet in the 
 series of papers \cite{gr1}-\cite{gr6} (see also a recent work \cite{cegarra}).  So our results shed light on Grillet theory. For instance we relate the commutative monoid (co)homology with Andr\'e-Quillen (co)homology of corresponding monoid algebra. For another application we mention   the Hodge decomposition for  commutative monoid  (co)homology which is an  immediatle consequence of our main result.

\section{ $\Gamma$-modules and commutative algebra (co)homology}\label{gc}
\subsection{Generalities on $\Gamma$-modules}\label{cnobili}

Let $K$ be a commutative ring with unit which is fixed in the whole paper.

For any integer $n\geq 0$, we let $[n]$ be the set $\lbrace 0,1,...,n\rbrace$ with
basepoint $0$. Let $\Ga$ be the  full subcategory of the category of pointed sets consisting of sets $[n]$.
 A \emph{left $\Ga$-module} is  a covariant functor from $\Ga$ 
 to the category of $K$-modules, while  a \emph{right
$\Ga$-module} is  a contravariant functor from $\Ga$ to the
category of $K$-modules. The category of all left $\Gamma$-modules is denoted by $\G$-$mod$, while the category of all right modules is denoted by $mod$-$\G$. It is well-known that the categories $\G$-$mod$ and $mod$-$\G$ are abelian categories with sufficiently many projective and injective objects. 
For any $n\geq 0$ one defines the left $\Gamma$module $\Gamma ^n$ by
$$\Gamma ^n(X)=K[X^n].$$
Here $K[S]$ denotes the free $K$-module  generated by a set $S$. It is a consequence of the Yoneda lemma that for any left $\Gamma$-module $F$ one has natural isomorphisms
$$\Hom_{\Gamma}(\Gamma^n,F)\ \cong\  F([n]) .$$
Therefore $\Gamma^n$,  $n\geq 0$, are projective generators of the category $\Gamma$-$mod$ .

\subsection{Hochschild and Harrison  (co)homology of $\Gamma$-modules} The definition of these objects are based on the following pointed maps (see  \cite{inv}).
For any $ 0\leq i\leq n+1,$ one defines a map $$\epsilon^i:[n+1]\to [n], \  0\leq i\leq n+1,$$
by $$\epsilon^i(j)=\begin{cases}j&j\leq i,\\ j-1& j>i\leq
n,\\0&j=i=n+1 .\end{cases}$$

For a left $\Ga$-module $F$ the Hochschild homology $HH_*(F)$ is defined as the homology of the chain complex
$$F([0]) \leftarrow F([1])\leftarrow  F([2])\leftarrow \cdots \leftarrow F([n])\leftarrow \cdots$$
where the boundary map $\partial:F([n+1])\to F([n])$ is given by $\sum_{i=0}^{n+1}(-1)^iF(\epsilon^i)$.

Quite similarly for a right $\Ga$-module $T$ one
defines the Hochschild cohomology  $HH^*(T)$ as the cohomology of
the following cochain complex
$$T([0])\to T([1])\to \cdots \to T([n])\to T([n+1])\to\cdots $$
where the coboundary map $\delta:T([n])\to T([n+1])$ is given by
$\delta=\sum_{i=0}^{n+1}(-1)^iT(\epsilon^i)$. 

We have the following obvios fact.

\begin{Le} \label{hh1} Let $F$ be  a  left $\Ga$-module, then $HH_0(F)=F([0])$ and
$$
HH_1(F)=\Coker(\partial:F([2])\to F([1])).$$
Similraly, if $T$ is  a  r $ight \Ga$-module, then $HH^0(T)=T([0])$ and
$$
HH^1(T)=\ker(\delta:T([1])\to T([2])).$$
\end{Le}

Let $S_n$ be the symmetric group on $n$ letters, it acts as a group
of automorphisms  on $[n]$. For integers $p_1,\cdots,p_k$ with
$n=p_1+\cdots +p_k$, we set
$$sh_{p_1,\cdots,p_k}=\sum  {\sf sgn}(\sigma)\sigma\in \Z[S_n]$$
where $\sigma\in S_n$ is running  over all
$(p_1,\cdots,p_k)$-shuffles. Each $sh_{p_1,\cdots,p_k}$ induces a
map $T([n])\to T([n])$, called the shuffle map. Let us denote by
$\tilde{T}_n$ the intersection of the kernels of all shuffle maps.
These groups form a subcomplex of the Hochschild complex, called
Harrison complex \cite{inv}. The groups $Harr^n(T)$, $n\geq 0$ are defined as
the cohomology of the Harrison complex.

By duality we have also Harrison and Hochschild homology of left $\Ga$-module.

The following is a theorem due to J.-L. Loday \cite{inv}. For alternative approach see  \cite{hodge}. 

\begin{Th}\label{Hodge} If  $K$
is a field of characteristic zero, then for any left $\Gamma$-module $F$ and right $\Gamma$-module $T$
there exist so called Hodge decompositions:
$$HH_n(F)\cong \bigoplus_{i=1}^nHH_n^{(i)} F), \ \ \ n>0,$$
$$HH^n(F)\cong \bigoplus_{i=1}^nHH^n_{(i)}(T), \ \ \ n>0,$$
for suitable defined $HH_n^{(i)} (F)$ and $HH^n_{(i)}(T)$. Moroeover, for $i=1$ one has
$$Harr_{n}(F)=HH_n^{(1)} (F), \ \ Harr^{n}(F)=HH^n_{(1)}(T), \ \ \ n>0.$$
\end{Th}
\subsection{Andr\'e-Quillen  (co)homology of $\G$-modules}  We recall some material from \cite{aq}. 

{\it A partition} $\ll =
(\ll_1,\cdots, \ll_k)$
is a sequence
of natural numbers $\ll_1\geq \cdots  \geq \ll_k\geq 1$. 
The sum of partition is given by
$s(\ll):=\ll_1+\cdots +\ll_k$, while the group $\Sigma (\ll)$ 
is a product of the corresponding symmetric groups 
$$\Sigma ({\ll}):=\Sigma_{\ll_1}\times \cdots \times \Sigma_{\ll_k}.$$
which is identified with the Young subgroup of $\Sigma _{s(\ll)}$. 
Let us observe that $\Sigma _n=Aut_{\G}([n])$ and therefore $\Sigma _n$
acts on $F([n])$ and $T([n])$ for any left $\Gamma$-module $F$ and right $\Gamma$-module $T$.

Let $$0\to F_1\to F\to F_2\to 0 $$
be an exact sequence of left $\Ga$-modules. It is called a {\it $\Y$-exact 
sequence} if for any partition $\ll$ with $s(\ll)=n$ the 
induced map 
$$F([n])^{\Sigma (\ll)}\to F_2([n])^{\Sigma (\ll)}$$ 
is surjective.  The class of  $\Y$-exact 
sequences is proper in the sense of MacLane \cite{homology}.

A left $\Gamma$-module $F$ is $\Y$-\emph{projective}, if the functor $\Hom_\Gamma(F,-)$ takes $\Y$-exact sequences to exact sequences. For example $S^n\Gamma^1$ is a $\Y$-projective \cite{aq}. Here $S^n$ denotes the $n$-th symmetric power. According to \cite{aq} for any left $\Gamma$-module $F$ there is a $Y$-exact sequence
$$0\to F_1\to F_0\to F\to 0$$
with $\Y$-projective $F$.  Hence one can take relative left
derived functors of the functor $HH_1:\Gamma$-$mod\to K$-$mod$. The values of these derived functors on a left $\Gamma$-module $F$ is denoted by ${\pi^\Y} _*(F)$. So by the definition the functors ${\pi^\Y}_*$ 
are uniquely defined (up to isomorphism) by the following properties

\begin{Le} There exist unique family of functors ${\pi^\Y}_n:\Gamma$-$mod\to K$-$mod$, $n\geq 0$, such that

i) ${\pi^\Y}_0(F)=HH_1(F)$.

ii) For any $\Y$-exact 
sequence 
$$0\to F_1\to F\to F_2\to 0 $$
there is a long exact sequence
$$\cdots \to { \pi^\Y}_{n+1}(F_2)\to {\pi^\Y}_n(F_1)\to \cdots \to {\pi^\Y}_0(F_1)\to
{ \pi^\Y}_0(F)\to {\pi^\Y}_0(F_2)\to 0.$$

iii) The functor ${\pi^\Y}_n$ vanishes on $\Y$-projective objects, $n\geq 1$.
\end{Le}

By a dual argument for any right $\Gamma$-module $T$ one obtains $K$-modules ${\pi_\Y}^*(T).$
\subsection{$\Gamma$-modules and commutative algebras} The classical Hochschild cohomology (as well as the Harrison or Andre-Quillen (co)homology) of commuttaive algebas  is a particular case of the cohomology of $\Gamma$-modules \cite{inv}, \cite{hodge}, \cite{aq}. We recall the corresponding results.
Let $R$ be a commutative $K$-algebra and  $A$ be an $R$-module. We have  a left and right $\Gamma$-modules
${\mathcal L}_*(R,A)$ and ${\mathcal L}^*(R,A)$ defined  on objects by
$${\mathcal L}^*(R,A)([n]):=\Hom(R^{\t n},A), \  \ \ {\mathcal L}_*(R,A)([n]):=R^{\t n}\t A.$$
For a pointed map $f:[n]\to [m]$, the action of  ${\mathcal L}^*(R,A)$ on
$f$ is given by
$$f^*(\psi)(a_1\t\cdots \t a_n)=b_{0}\psi(b_1\t\cdots \t b_m)$$
 while for the functor ${\mathcal L}_*(R,A)$ one has:
$$f_*(a_0\t\cdots\t a_n)=b_0\t\cdots\t b_m$$
where  $b_j=\prod _{f(i)=j}a_i, \ j=0,\cdots ,n.$

Then one has \cite{inv}:
$$HH_*({\mathcal L}_*(R,A))=HH_*(R,A), \ \ HH^*({\mathcal L}^*(R,A))=HH^*(R,A),$$
$$Harr_n({\mathcal L}_*(R,A))=Harr_m(R,A), \ \ Harr^n({\mathcal L}^*(R,A))=Harr^m(R,A).$$
where $HH_*(R,A)$ and $Harr_*(R,A)$ (resp. $HH^*(R,A)$, $Harr^*(R,A)$ ) are the Hochschild and Harrison (co)homology of $R$ with coefficients in $A$.

By \cite{aq}  a similar result is aslo true for Andr\'e-Quillen  (co)homology of commutative rings. In order to state this result,
let us first  recall  the  definition of  Andr\'e-Quillen (co)homology  \cite{Q}.

 Let $\bf SCR$ be category of simplicial commutative rings and let $\bf SS$ be the category of simplicial
 sets and let
 $U:{\bf SCR} \to \bf SS$ be a forgetful functor. According to \cite{ha}
 there is a unique closed model category structure on the category $\bf SCR$ such
 that a morphism $f:X_*\to Y_*$ of $\bf SCR$ is
weak equivalence (resp. fibration) if $U(f)$ is a weak equivalence
(resp. fibration) of simplicial sets. A simplicial commutative ring
 $X_*$
is called \emph{free} if each $X_n$ is a free commutative ring
 with a
base $S_n$, such that degeneracy operators $s_i:X_n\to X_{n+1}$ maps
$S_n$ to $S_{n+1}$, $0\leq i\leq n$. Thanks to \cite{ha} any free
simplicial commutative ring
 is cofibrant and any cofibrant object is a retract of a free
 simplicial commutative ring.

We let $C^*(V^*)$ be the cochain complex associated to a
cosimplicial abelian group $V^*$. Let $R$  be a commutative ring and
let $A$ be an $R$-module.  Then the Andr\'e-Quillen cohomology of
$R$ with coefficients in $A$ is defined by (see \cite{Q}):
$${\sf D}^*(R,A):=H^*(C^*(\der({P_*},A))),$$
where $P_*\to R$ is a cofibrant replacement of the ring $R$
considered as a constant simplicial ring and $\der$ denotes the
abelian group of all  $K$-derivations.

The Andr\'e-Quillen homology of
$R$ with coefficients in $A$ is defined by
$${\sf D}_*(R,A):=H_*(C_*( A\otimes _{P_*}\Omega^1_{P_*})),$$
where $\Omega^1_R$ is the K{\"a}hler differentials of  a commutative $K$-algebra  $R$.

The main result of \cite{aq} claims that there are natural isomorphisms
$${\pi^\Y}_*({\mathcal L}_*(R,A))=D_*(A,M), \ \  \ {\pi_\Y}^*({\mathcal L}^*(R,A))=D^*(A,M).$$

\section{The category $\H(C)$ associated to a commutative monoid $C$} \label{DDC}

\subsection{ Definition}
Let $C$ be a commutative monoid. Define the category $\H(C)$ as follows. 
Objects of $\H(C)$ are elements of $C$. A morphism   from an
element $a\in C$ to an element $b$ is a pair $(c,a)$ of elements of
$C$ such that $b=ca$. To simplify notations we write $a\xto{c} ac$ for a morphism $(a,c):a\to b=ac$.
If $:a\xto{c} ac$ and $ac\xto{d} acd$ are
morphisms in $\H(C)$, then the composite of these morphisms in ${\H}(C)$ is $a\xto{cd} acd$. 

It is clear that $1\in C$ is an initial object of $\H(C)$.

 A \emph{left  $\H(C)$-module} is  a covariant functor $A: \H(C)\to \Ab$, similarly a 
\emph{right  $\H(C)$-module} is  a contravariant functor $A: \H(C)^{op}\to \Ab$. 

We let $\H(C)$-$mod$ be  the category of left $\H(C)$-modules, while  $mod$-$\H(C)$ denotes the category of right $\H(C)$-modules. If
$M$ is a left $\H(C)$-module. Then the value
of $M$  on the element $a\in C$ (considered as object of $\H(C)$) is denoted by  $M(a)$. Moreover if $a,b,c\in C$ and $b=ca$,
then we have an induced map $c_*:M(a)\to M(b)$, with obvious
properties $1_*=\id$ and $(c_1c_2)_*=c_{1*}c_{2*}$. 

Quite similarly, if $N$ is a right $\H(C)$-module, then the value
of $N$ on the element $a\in C$ is denoted by $N(a)$. Moreover if $a,b,c\in C$ and $b=ca$,
then we have an induced map $c^*:N(b)\to N(a)$, with obvious
properties $1^*=\id$ and $(c_1c_2)^*=c_2^*c_{1}^*$. 

The categories $\H(C)$-$mod$ and  $mod$-$\H(C)$ are abelian categories
with enough projective and injective objects. For any element $a$ of $C$ we let $C^a$  and $C_a$  be respectively the left   and  right  $\H(C)$-modules defined by
$$ C^a(x)=\bigoplus _{c\in (x:a)}\Z,  \  C_a(x)=\bigoplus _{c\in (a:x)}\Z.$$
Here for elements $a,b\in C$ we let $(b:a)$ be the set of all elements $c\in C$ such that $b=ac.$
By Yoneda lemma for any left $\H(C)$-module $A$ and for any right  $\H(C)$-module  $B$ one has isomorphisms
$$Hom_{{\H}(C)}(C^a,A)\cong A(a), \ \  \   \ \ \ \ \  \ \ \ Hom_{{\H}(C)}(C_a,B)\cong B(a)).$$
It follows that $C^a$,  $a\in C$ form a family of projective generators of the category $\H(C)$-$mod$. Similarly   $C_a$,  $a\in C$ form a family of projective generators of the category $mod$-$\H(C)$.

Let  $N$ be a right $\H(C)$-module and $M$ be a left $\H(C)$-module. We let $N\otimes _{\H(C)} M$ be the abelian group generated by elements of the form $x\otimes y$, where $x\in N(a)$, $y\in M(a)$, $a\in M$. These elements are subject to the following relations
$$(x_1+x_2)\otimes y=x_1\otimes y+x_2\otimes y,$$
$$x\otimes (y_1+y_2)=x\otimes y_1+x\otimes y_2,$$
$$c^*(z)\otimes y=z\otimes c_*(y).$$
Here $c\in M$, $x,x_1,x_2\in N(a)$, $y,y_1,y_2\in M(a)$, $z\in N_{ca}$. Then one has $$N\otimes _{{\H}(C)}C^a\cong N(a), \ \ \ C_a\otimes _{{\H}(C)}M\cong M(a).$$

If $f:C\to C'$ is a homomorphism of monoids, then $f$ induces a functor ${\H}(f):{\H}(C)\to {\H}(C')$ in an obvious way. Thus for any left ${\H}(C')$-module $M$ one has a left ${\H}(C)$-module $f^*(M)$, which is given by
$$f^*(M) (i)=M(f(i)).$$
In this way one obtains a functor $f^*$ from the category of (left or right) modules over $\H(C')$ to the category of modules over $\H(C)$.

\subsection{$K[C]$-modules and $\H(C)$-modules} \label{2.3} We let $K[C]$ be the monoid algebra of the monoid $C$. Any
$K[C]$-module $A$ gives rise to the left $\H(C)$-module
$j^*(A)$ 
which is defined by
 $$j^*(A)(a)=A$$ 
and  for $b=ca$,
the induced morphisms
$$A=j_*(A)(a)\xto{c_*} j_*(A)(b)=A$$
is simply the multiplication by $c$. 

If $M$ is a left $\H(C)$-module, we let $j_*(M)$ be the following $K[C]$ module. As an abelian group one has
$$j_*(M)=\bigoplus_{x\in C} M(x),$$
The module structure is defined as follows: for $x\in C$, $a\in M(x)$ and  $c\in C$ one has $$ci_x(a)=i_{cx}(c_*(a)).$$ Here $i_x$ is the canonical inclusion $M(x)\to j_*(M)$, $x\in C$. 
\begin{Le} The functor  $j_*$ is a left adjoint functor  to $j^*$. 
\end{Le}
\begin{proof} For a left $\H(C)$-module $M$ and a left $K[C]$-module $A$ an elements $$\xi\in\Hom_{\H(C)}(M, j^*(A))$$ is given by the family of $K$-module homomorphisms $\xi_a:M(a)\to A$, $a\in A$ such that for any $c\in C$ the following 
$$\xymatrix{M(a)\ar[r]^{\xi_a} \ar[d]_{c_*}& A\ar[d]^c\\ M(ac)\ar[r]_{\xi_{ac}}&A}$$
The homomorphisms $\xi_a$, $a\in C$ defines a homomorphism of $K$-modules $$\hat{\xi}:j_*(M)=\bigoplus_{a\in C}M(a)\to A$$
which clearly is $K[M]$-homomorphism. So, $\xi\mapsto \hat{\xi}$ gives rise a homomorphism $\Hom_{\H(C)}(M, j^*(A))\to \Hom_{K[C]}(j_*(M), A)$ which is obviously an isomorphism.
\end{proof}

\subsection{Derivations, differentials and (co)homology in the theory of commutative algebras}

Let $C$ be a commutative monoid and let $M$ be a left ${\H}(C)$-module.
A \emph{derivation}
$\delta:C\to M$  of $C$ with values in $M$ is a function which assigns to
each element $a\in C$  an element $\delta(a)\in M(a)$, such that
$$\delta(ab)=a_*(\delta(b))+b_*(\delta(a)).$$ The abelian group of all derivations of
$C$ with values in $M$ is denoted by $\der(C,M)$.

We claim that there exist a universal derivation. In fact we construct a left $\H(C)$-module $\Omega_C$, called \emph{differentials} of a monoid $C$. It is a 
left $\H(C)$-module generated by symbols $da\in \Omega_C(a)$ one for every element $a\in C$, subject to relations
$$d(ab)=a_*(d(b))+b_*(d(a))$$
for every $a$ and $b\in C$. It follows from the construction that $a\mapsto da$ is a derivation, which is clearly universal one, in the sense that for any derivation $\delta:C\to,M$ there is a unique homomorphism of $\H(C)$-modules $\delta^*:\Omega_C\to M$ such that $\delta(a)=\delta^*(da)$. Thus  for any left $\H(C)$-module $M$ one has a canonical isomorphism
$$\der(C,M)\cong Hom_{\H(C)}(\Omega_C,M).$$

\begin{Le}  \label{22}  One has an isomorphism of $K[C]$-modules
$$j_*(\Omega_C)=\Omega_{K[C]}^1$$
Here $j_*:{\H}(C)$-$mod\to K[C]$-$mod$ is the functor constructed in Section \ref{2.3} and $\Omega_{K[C]}^1$ is the  K\"ahler differentials of the $K$-algebra $K[C]$.
\end{Le}
\begin{proof} Let $A$ be a $K[C]$-module. Then we have
$$\der(C,j^*(A))=Hom_{\H(C)}(\Omega_C,j^*(A))=Hom_{K[C]}(j_*(\Omega),A).$$
On the other hand
$$\der(C,j^*(A))= \der(K[C],A)=Hom_{K[C]}(\Omega^1_{K[A]},A)$$
and the result follows from the Yoneda lemma.
\end{proof}

\subsection{The case  $C=\mathbb{N}$}\label{t}
If $C$ is the free abelian monoid with a generator $t$, then a left ${\H}(C)$-module is nothing but a diagram of abelian groups
$$M= (M_0\xto{t} M_1\xto{t} M_2\xto{t} M_3\xto{t}\cdots \,)$$
In particular the projective object $C^n$ corresponds to the diagram
$$0\to 0\to\cdots \to 0\to\Z\xto{1}\Z\xto{1}\cdots$$
where the first nontrivial group appears at the place $n$.

Quite similarly a right ${\H}(C)$-module is nothing but a diagram of abelian groups
$$N=(\,\cdots \xto{t} N_3\xto{t} N_2\xto{t} N_1\xto{t}N_0).$$
In particular the projective object $C_n$ corresponds to the diagram
$$\cdots\to 0\to 0\to \Z\xto{1}\Z\xto{1}\cdots \xto{1} \Z$$
where the first nontrivial group appears at the place $n$.

One easily observes that for any left ${\H}(C)$-module $M$ one has an isomorphism
$$\der(C,M)\cong M_1$$
which is given by $\delta\mapsto \delta(t)$. This follows from the fact that $\delta(t^n)=nt^{n-1}\delta(t).$ Thus
$$\Omega_C=C^1=(0\to \mathbb{Z}\xto{1}\mathbb{Z}\xto{1}\mathbb{Z}\xto{1}\cdots \,).$$
\subsection{Product of two monoids} Let $C$ be a product of two monoids: $C=C_1\times C_2$. Then ${\H}(C)={\H}(C_1)\times {\H}(C_2)$. Assume $M_1$ and $M_2$ are (say left) $\H(C_1)$ and $\H(C_2)$-modules respectively. Then one can form a $\H(C)$-module $M_1\boxtimes M_2$ as follows:
$$M_1\boxtimes M_2(x_1,x_2)=M_1(x_1)\otimes M_2(x_2).$$
\begin{Le}\label{2.4} For any elements $c_1\in C$ and $c_2\in C_2$, one has
$$C^{(c_1,c_2)}=C^{c_1}\boxtimes C^{c_2}$$
and 
$$C_{(c_1,c_2)}=C_{c_1}\boxtimes C_{c_2}.$$
\end{Le}
\begin{proof} By definition one has
\begin{align*}C^{c_1}\boxtimes C^{c_2}(x_1,x_2)&=C^{c_1}(x_1)\otimes C^{c_2}(x_2)\\
&=\left(\bigoplus_{a_1\in C_1; a_1c_1=x_1}\Z \right)\otimes  \left(\bigoplus_{a_2\in C_2; a_2c_2=x_2}\Z \right )\\
& =\bigoplus_{(a_1,a_2)(c_1,c_2)=(x_1,x_2)}\Z \\
&=C^{(c_1,c_2)}(x_1,x_2).
\end{align*}
Similarly for the second isomorphism.
\end{proof}

We have homomorphisms
$$ \iota_1:C_1\to C, \  \iota(c_1)=(c_1,1), \  \ \ \iota_2:C_2\to C, \  \iota(c_2)=(1,c_2).$$
For any left ${\H}(C)$-module $M$ we set $$M^{(1)}=\iota_1^*(M), \  \  M^{(2)}=\iota_2^*(M).$$

\begin{Le} For any left ${\H}(C)$-module $M$ one has
$$\der(C,M)\cong \der(C_1,M^{(1)})\oplus \der(C_1,M^{(2)}).$$
\end{Le}

\begin{proof} This easily follows from the fact $(c_1,c_2)=(c_1,1)(1,c_2).$
\end{proof}

We also have projections $\pi_1:C\to C_1$ and s $\pi_2:C\to C_2$, given respectively by $pi_i(c_1,c_2)=c_i$, $i=1.2$.
\begin{Le} For any left $\H(C_i)$-module $X_i$, $i=1,2$ and any  left $\H(C)$-module  $M$, one has isomorphisms
$$\Hom_{\H(C)}(\pi_1^*X_1, M)\cong \Hom_{\H(C_1)}(X_1, M^{(1)})$$
and $$  \Hom_{\H(C)}(\pi_2^*X_2, M)\cong \Hom_{\H(C_2)}(X_2, M^{(2)}).$$
\end{Le}
\begin{proof} Let $\eta\in \Hom_{\H(C)}(\pi_1^*X_1, M).$ Thus $\eta$ is a collection of homomorphisms of $K$-modules
$$\eta_{(a_1,a_2)}:X_{a_1}\to M_{(a_1,a_2)}$$
such that for any elements $c_1\in C_1$, $c_2\in  C_2$ the following diagram commutes
$$\xymatrix{X_{a_1}\ar[rr] ^{\eta_{(a_1,a_2)}}
\ar[d]_{c_{1*}} &&M_{(a_1,a_2)}\ar[d]^{(c_1,c_2)_*}\\
X_{a_1c_1}\ar[rr]_{\eta_{(a_1c_1,a_2c_2)}}&& M_{(a_1c_1,a_2c_2)}
}$$
it follows that $\eta_{(a_1,a_)}=(1,a_2)_*\circ \eta_{(a_1,1)}$. 
It is clear that the familly of homomorphisms $\eta_{a_1,1}$, $a_1\in C_1$ defines the morphism $\hat{\eta}\in \Hom_{\H(C_1)}(X_1, M^{(1)})$ and the previous equality shows that $\eta\mapsto \hat{\eta}$ is really a bijection.
\end{proof}

\begin{Co} If $C=C_1\times C_2$, then 
$$\Omega_C=\pi_1^*\Omega_{C_1}\oplus \pi_2^*\Omega_{C_1},$$
where $\pi_i:C\to C_i$, $i=1,2$ is the canonical projection. 
\end{Co}

\begin{proof} For any $\H(C)$-module  $M$ one has
\begin{align*}
\Hom_{\H(C)}(\Omega_C,M)&=\der(C,M)\\
&=\der(C_1,M^{(1)}) \oplus \der(C_2,M^{(2)})\\
&=\Hom_{\H(C_1)}(\Omega_{C_1},M^{(1)})\oplus \Hom_{\H(C_2)}(\Omega_{C_2},M^{(2)})\\
&=\Hom_{\H(C)}(\pi_1^*\Omega_{C_1},M)\oplus \Hom_{\H(C)}(\pi_2^*\Omega_{C_2},M)\\
&=
\Hom_{\H(C)}(\pi_1^*\Omega_{C_1}\oplus \pi_2^*\Omega_{C_2},M)
\end{align*}
and the result follows from the Yoneda lemma.
\end{proof}

\section{Commutative monoid (co)homology and $\G$-modules}\label{cmc}
\subsection{$\Gamma$-modules related to monoids} Let $C$ be a commutative monoid and let $M$ be a left $\H(C)$-module. Define  a right $\Gamma$-module 
   ${\sf G}^*(C,M)$ as follows. On objects it is igiven  by
$${\sf G}^*(C,M)([n])=\prod_{(a_1,\cdots,a_n)\in C^n}M_{a_1\cdots
a_n}.$$ Thus $\eta\in {\sf G}(C,M)([n])$ is a function which assigns
to any $n$-tuple of elements $(a_1,\cdots ,a_n)$ of $C$ an element
$\eta(a_1,\cdots,a_n)\in M_{a_1\cdots a_n}$. Let  $f:[n]\to [m]$ be
a pointed map and $\xi\in {\sf G}(C,M)([m])$. Then the function
$f^*(\xi)i\in {\sf G}(C,M)([n])$ is given by
 $$f^*(\xi)(a_1,\cdots,a_n)=b_{0*}(\xi(b_1,\cdots,b_m)).$$ 

Quite similarly, let $N$ be  a right $\H(C)$-module. Define  left $\Gamma$-module ${\sf G}_*(C,N)$  as follows. On objects it is given by
$${\sf G}_*(C,N)([n])=\bigoplus_{(a_1,\cdots,a_n)\in C^n}N(a_1\cdots
a_n).$$
In order, to extend the definition on morphism, we let $$\iota_{(a_1,\cdots,a_n)}: N(a_1\cdots
a_n)\to {\sf G}_*(C,N)([n])$$
be the canonical inclusion. Let $f:[n]\to [m]$ be
a pointed map. Then the homomorphism
$$f_*:{\sf G}_*(C,N)([n])\to {\sf G}_*(C,N)([m])$$
is defined by
$$f_*\iota_{(a_1,\cdots,a_n)}(x)=\iota_{(b_1,\cdots,b_m)}((b_0)_*(x)),$$
where $x\in N(a_1\cdots a_n)$ and $$b_j=\prod _{f(i)=j}a_i,  \ \  \ j=0,\cdots ,n.$$ 
Here we used the convention that $b_0=1$ provided $f^{-1}(\{0\})=\{0\}$.

\begin{Le}\label{3.1} Let $C=\mathbb{N}$ be a free commutative monoid with a generaor $t$, and let $C_n$ be the standard projective right $\H(C)$-module, $n\geq 0$, see Section \ref{t}. Then one has an isomorphism of left $\Gamma$-modules
$${\sf G}_*(C,C_n)=\bigoplus_{k=0}^nS^k\circ \Gamma^1$$
In particular, ${\sf G}_*(C,C_n)$ is ${\sf Y}$-projective.
\end{Le}

\begin{proof} Since $\Gamma^1([m])$ is a free $K$-module spanned on $x_1,\cdots, x_m$, it follows that $S^k\circ \Gamma^1([m])$ is a free $K$-module spanned by all monomials of degree $k$ on the variables $x_1,\cdots, x_m$. On the other hand we have
$${\sf G}_*(C,C_n)([m])=\bigoplus_{k=0}^n\bigoplus_{n_1+\cdots+n_m=k}\Z.$$
To see the expected isomorphism, it is enough to  assigne to  a basis element of $\bigoplus_{n_1+\cdots+n_k=m}\Z$ the monomial $x_1^{n_1}\cdots x_m^{n_m}$.
\end{proof}

\begin{Le}\label{3.2} Let $C=C_1\times C_2$ be product of two monoids and $N_i$ be right $\H(C_i)$ modules, $i=1,2$. Then one has
$${\sf G}_*(C,N_1\boxtimes N_2)={\sf G}_*(C_1,N_1) \otimes {\sf G}_*(C_2, N_2).$$
\end{Le}

\noindent The proof is straightforward.

\begin{Co}\label{coproj} Let $C$ be  a finitely generated free commutative monoid and let 
$N$ be a projective object in the category of right $\H(C)$-modules. Then ${\sf G}_*(C,N)$ is a ${\sf Y}$-projective left $\Gamma$-module.
\end{Co}

\begin{proof} Since, any projective object is a retract of a direct sum of standard projective modules $C^c$, it is enough to restrict ourself with the case when $C=C^c$. Assume $C=\mathbb{N}^k$. We will work by induction on $k$. If $k=1$, then the result was already  established, see Lemma \ref{3.1}. Rest follows from  Lemma 2.4 and Lemma \ref{3.2} and the fact that tensor product of two ${\sf Y}$-projective objects is ${\sf Y}$-projective see \cite{aq}.
\end{proof}


\subsection{Homology and cohomology  of commutative monoids} 
 Let $\bf CM$
be the category of all commutative monoids and let ${\bf SCM}$ be
the category of all simplicial commutative monoids. There is a
forgetful functor $U':{\bf SCM} \to \bf SS$. By \cite{ha} there is a
unique closed model category structure on the category $\bf SCM$
such that a morphism $f:X_*\to Y_*$ of $\bf SCM$ is  a weak
equivalence (resp. fibration) if $U'(f)$ is a weak equivalence
(resp. fibration) of simplicial sets. A simplicial commutative
monoid $X_*$ is called \emph{free} if each $X_n$ is a free
commutative monoid with a base $Y_n$, such that degeneracy operators
$s_i:X_n\to X_{n+1}$ maps $Y_n$ to $Y_{n+1}$, $0\leq i\leq n$.
According to \cite{ha} any free simplicial commutative monoid is
cofibrant and any cofibrant object is a retract of a free simplicial
commutative monoid.

If $C'\to C$ is a morphism of commutative monoids then it gives rise
to a functor ${\H}(C')\to {\H}(C)$, which allows us to consider
any left or right $\H(C)$-module as a module over $\H(C')$. In particular 
if $P_*\to C$ is an augmented
simplicial monoid and $M$ is a left $\H(C)$-module, one
 can considered $M$ as a left 
$\H(P_k)$-module, for all $k\geq 0$. The same holds for right $\H(C)$-modules.

Let  $M$ be a left $\H(C)$-module.   Then the Grillet
 cohomology of $C$ with coefficients
in $M$ is defined by
$${\sf D}^*(C,M):=H^*(C^*(\der({P_*},M))),$$
where $P_*\to C$ is a cofibrant replacement of the monoid $C$
 considered as a constant simplicial monoid.
 
 Let  $N$ be a right $\H(C)$-module.   Then the Grillet
 homology of $C$ with coefficients
in $N$ is defined by
$${\sf D}_*(C,N):=H_*(C_*(\Omega_{P_*}\otimes_{\H(P_*)}N))),$$
where $P_*\to C$ is a cofibrant replacement of the monoid $C$
 considered as a constant simplicial monoid.

The  definition of the cohomology essentially goes back to Grillet (see
\cite{gr1}-\cite{gr4}), but the definition of the Grillet homology is new.

 By comparing the definition we obtain  the following basic fact, which 
is missing  in (see \cite{gr1}-\cite{gr4}).
\begin{Le} Let $C$ be a commutative monoid and $A$ be a $K[C]$-module.
Then one has the isomorphisms:
$${\sf D}^*(C,j^*(A))\cong {\sf D}^*(K[C],A),$$
$${\sf D}_*(C,j^*(A))\cong {\sf D}_*(K[C],A).$$
\end{Le}

\begin{proof} The isomorphism in the dimension zero is obvios one, compare with Lemma \ref{22}. Rest
follows from the fact that if $P_*\to C$ is a cofibrant replacement
of $C$ in the category $\bf SCM$, then $K[P_*]\to K[C]$ is a
cofibration replacement of $K[C]$ in the category $\bf SCR$.
\end{proof}

\subsection{The main Theorem}
Now we are in the situation to state our main theorem, which relates Grillet (co)homology of the monoid $M$ with the Andre-Quillen (co)homology of the $\Ga$-modules ${\sf G}_*(C,N)$ and ${\sf G}^*(C,M)$.

\begin{Th}\label{grAQ} Let $C$ be a commutative monoid, $M$ be a left and $N$ be a right $\H(C)$-modules. Then one has the following isomorphisms
$${\sf D}^*(C,M)={\pi_\Y}^*({\sf G}^*(C,M)),$$
$${\sf D}_*(C,M)={\pi^\Y}_*({\sf G}_*(C,N)).$$
\end{Th}

The proof is based on several steps. The idea is to reduce the theorem to the case when $M$ is a free commutative monoid with one generators. In this case, theorem is proved using direct computation. We give proof only for homology, a dual argument works for cohomology.

We need some lemmata.
\begin{Le}\label{nuli}
Let $C$ be  a commutative monoid,  $N$ be a right $\H(C)$-module. Then  one has  natural isomorphisms
$$HH_1({\sf G}_*(C,N))\cong N\t_{\H(C)}\Omega_{C}.$$
\end{Le}
 \begin{proof}
Thanks to Lemma \ref{hh1} one has
$HH_1({\sf G}_*(C,N)$ is isomorphic to the cokernel of the map
$$\partial: \bigoplus_{a,b\in C}N(ab)\to \bigoplus _{a\in C} N(a)$$
As usual, we let $i_a:N(a)\to \bigoplus _{a\in C} N(a)$ be the canonical inclusion. For an element $x\in N(a)$, the  class of $i_a(x)$ in $HH_1({\sf G}_*(C,N))$ is denoted by ${\sl cl}(a;x)$.
Then $${\sl cl}(a;x)\mapsto x\otimes da$$
defines the isomorphism $HH_1({\sf G}_*(C,N))\cong N\t_{\H(C)}\Omega_{C}.$ 
 \end{proof} 

\begin{Le}\label{sizuste}
Let $C$ be a commutative monoid and let
 $$0\to N_1\to N\to N_2\to 0$$
ibe a short exact sequence of right $\H(C)$-modules, then
$$0\to {\sf G}_*(C,N_1)\to {\sf G}_*(C,N)\to {\sf C}_*(C,N_2)\to 0$$
is a $\Y$-exact sequence of left $\Gamma$-modules.
\end{Le}

\begin{proof} 
 For a partition $\ll$ of $n$ and a set $P$
we denote by $P^\ll$  theset of orbits of the cartesian product
$P^n$ under the action of the group $\Sigma(\ll)\subset \Sigma_n$.
In particular we have a set $C^\ll$. For any element $\mu\in
C^{\ll}$ we put $N_\mu:=N({a_1\cdots a_n})$, where
$(a_1,\cdots,a_n)\in \mu$. Since
$${\sf G}_*(C,N)([n])^{\Sigma(\ll)}=\bigoplus_{\mu\in C^\ll}N_\mu$$
the result follows.
\end{proof}

By the same argument we have also the following.
\begin{Le}
 Let $f:D\to C$ be a surjective homomorphism of commutative
monoids, then for any  right  $\H(C)$-module  $N$ 
 the
induced morphism of left  $\Ga$-modules
$${\sf G}_*(D,M)\to
{\sf G}_*(C,M)$$ is a $\Y$-epimorphism. 

\end{Le}
\begin{proof} In the notation of the proof Lemma \ref{sizuste} the map $D^\lambda\to C^\lambda$ is surjective and the result follows.
\end{proof}

\begin{Le}\label{rezolventa}
Let $\ee:X_*\to C$ be a simplicial resolution in the
category of commutative monoids and  $N$ be a  right  $\H(C)$-module.
 Then 
the associated chain complexes of the
simplicial left $\Ga$-module ${\sf G}_*(X_*,N) \to {\sf G}_*(C,N)$ is a
$\Y-$resolution. 
\end{Le}

\begin{proof}  Since $X_*^\lambda \to C^\lambda$ is a weak equivlence the result follows.

\end{proof} 
\subsection{Proof of Theorem \ref{grAQ}} Thanks to Lemma \ref{nuli} Theorem is true for $i=0$. First we consider the  case,
when $C=\mathbb{N}$ is the free commutative monoid with a generator $t$.
In this case ${\sf D}_i(C,-)=0$, if $i>0$. On the other hand ${\sf G}_*(C,C_n)$ is $\Y$-projective thanks to Lemma \ref{3.1}. Therefore  the result is true in this case. It follows from Lemma \ref{2.4}, Lemma \ref{3.2} and Lemma 4.2 of \cite{aq}  that the result is true if $C$ is  a free commutative monoid and $N$ is projective. By Lemma \ref{sizuste} the functor 
$\pi^\Y_*({\sf G}_*(C,-))$ assignes the long exact sequence to a
short exact sequence of right $\H(C)$-modules. Therefore 
we can consider such an exact sequence
associated to a short exact sequence of right $\H(C)$-modules
$$0\to N_1\to F\to N\to 0$$
with projective $F$. Since the result is true if $i=0$ one obtains by
induction on $i$, that $AQ_i({\sf G}_*(C,-))=0$ provided $i>0$ and $C$ is free commutative monoid.

Now consider the  general case. Let $P_*\to C$ be a free simplicial 
resolution in the category of commutative monoids. Then we have
$$N\t_{\H(P)_*} \Omega \cong {\pi^\Y}_0({\sf G}_*(C,N))$$
Thanks to Lemma \ref{rezolventa} 
$C_*({\sf G}_*(P_*,N))\to {\sf G}_*(C,N)$ is a $\Y$-resolution
consisting with $AQ_*$-acyclic objects and the result follows.

\subsection{Applications}  Let $C$ be  a commutative monoid, $M$ be a left $\H(C)$-module and $N$ be a right $\H(C)$-module.  For the $\Gamma$-modules ${\sf G}_*(C,N)$ and ${\sf G}^*(C,M)$ one can apply the reach theory of functor homology developped in \cite{inv}, \cite{hodge}, \cite{aq}. For example if one applies the Hochschild and Harrison theoryies one gets groups $HH_*(C,N)$, $Harr_*(C,N)$ and $HH^*(C,M)$, $Harr^*(C,M)$. 
Comparing with definitions one sees that  $HH^*(C,M)$ is nothing but  Leech cohomology \cite{leech}.
If $K$ is a field of characteristic zero, then we have 
  $$D_*(C,N)=Harr_{*+1}(C,N), \   \ \  D^*(C,M)=Harr^{*+1}(C,M)$$
this follows from general results valid for arbitrary   $\Gamma$-modules \cite{hodge},\cite{aq}.
In particular this solves the cocycle problem for Grillet  cohomology \cite{gr5}. By theorem \ref{Hodge}  we also obtain that the Grillet cohomology is a direct summand of  Leech cohomology.

\end{document}